\documentclass[11pt,reqno]{amsart}
\usepackage[all]{xy}
\usepackage{amssymb}
\usepackage{amsthm}
\usepackage{amsmath,mathtools}
\usepackage{amscd,enumitem}
\usepackage{verbatim}
\usepackage{eurosym}
\usepackage{float}
\usepackage{color}
\usepackage{dcolumn}
\usepackage[mathscr]{eucal}
\usepackage[all]{xy}
\usepackage{bbm}
\usepackage[textheight=8.5in, textwidth=6.7in]{geometry}
\newtheorem*{conj*}{Conjecture}
\newtheorem{theorem}{Theorem}[section]

\theoremstyle{definition}

\theoremstyle{plain}

\newtheorem{lemma}[theorem]{Lemma}

\newtheorem{prop}[theorem]{Proposition}

\newtheorem{conjecture}{Conjecture}
\newtheorem{definition}[theorem]{Definition}

\newtheorem{corollary}[theorem]{Corollary}

\newcommand{\legendre}[2]{\genfrac{(}{)}{}{}{#1}{#2}}

\renewcommand{\pmod}[1]{\,\,({\rm mod}\,\,{#1})}

\numberwithin{equation}{section}

\newtheoremstyle{example}
  {\topsep}   
  {\topsep}   
  {\normalfont}  
  {0pt}       
  {\bfseries} 
  {.}         
  {5pt plus 1pt minus 1pt} 
  {}          
\theoremstyle{example}

\def\({\left(}
\def\){\right)}

\def\lp{\left(}
\def\rp{\right)}

\usepackage{centernot}

\title{On Ramanujan-Type Congruences for Multiplicative Functions}
\author{William Craig}
\address{Department of Mathematics\\University of Virginia\\ Kerchof Hall 112\\ 141 Cabell Drive \\ Charlottesville, VA 22903}
\email{wlc3vf@virginia.edu}

\author{Mircea Merca}
\address{Department of Mathematics\\University of Craiova\\Craiova, 200585 Romania}
\email{mircea.merca@profinfo.edu.ro}


\begin{document}

\maketitle

\begin{abstract}
The study of Ramanujan-type congruences for functions specific to additive number theory has a long and rich history. Motivated by recent connections between divisor sums and overpartitions via congruences in arithmetic progressions, we investigate the existence and classification of Ramanujan-type congruences for functions in multiplicative number theory.
\end{abstract}

\section{Introduction}

Let $p(n)$ denote the number of partitions of the positive integer $n$, that is, the number of ways to represent $n$ as a sum of non-increasing positive integers. The function $p(n)$ has a rich history in combinatorics and number theory, including both analytic and arithmetic structure. Here, we focus on the body of work stemming from Ramanujan's brilliant observation \cite{Ramanujan1,Ramanujan3} that the function $p(n)$ satisfies certain congruence relations, given by
\begin{align*}
p(5n+4) \equiv 0 \pmod{5}, \\
p(7n+5) \equiv 0 \pmod{7},
\end{align*}
and
\begin{align*}
p(11n+6) \equiv 0 \pmod{11}.
\end{align*}
Ramanujan also knew congrunces modulo higher powers of primes, for instance $p(25n + 24) \equiv 0 \pmod{25}$ \cite{Ramanujan2}. Ramanujan also made the striking conjecture that there are ``no equally simple properties for any moduli involving primes other than these three'' \cite{Ramanujan2}, that is, the only congruences of the form $p(\ell n + b) \equiv 0 \pmod{\ell}$ for $\ell$ prime are the three identified by Ramanujan. This conjecture was proven by Ahlgren and Boylan \cite{AhlgrenBoylan}, and so congruences of the form $p(\ell n + b) \equiv 0 \pmod{\ell}$ are fully classified. Such classification results are natural and interesting questions in the broader theory of additive number theory. Here, we focus on the existence and classification of similar congruences for functions other than $p(n)$. To fix terminology, if $f : \mathbb{N} \to \mathbb{Z}$ is any arithmetic function, we say that $f(n)$ satisfies a {\it Ramanujan-type congruence} modulo a prime power $p^k$ if there is some arithmetic progression $An+B$ such that for all $n \geq 0$, we have
\begin{align*}
f(An+B) \equiv 0 \pmod{p^k}.
\end{align*}
The theory of Ramanujan-type congruences of arithmetic functions now reaches well beyond $p(n)$ to many related combinatorial functions, which include $k$-colored partitions \cite{BGRT} and overpartitions, to which we now turn our attention. Recall that an overpartition of the positive integer $n$ is an ordinary partition of $n$ where the first occurrence of parts of each size may be overlined \cite{Corteel}, and let let $\overline p(n)$ be the number of overpartitions of $n$, as in \cite{Corteel}. For example, the overpartitions of 3 are $(3), (\overline{3}), (2,1), (\overline{2},1), (2,\overline{1}), (\overline{2},\overline{1}), (1,1,1)$ and $(\overline{1},1,1)$, and so $\overline p(3) = 8$. In analogy with $p(n)$, many authors have found various Ramanujan-type congruences for the overpartition function $\overline p(n)$, using tools such as dissection formulas for theta functions \cite{Fortin,Hirschhorn} and the theory of modular forms \cite{Chen15}. We note that first congruences of this type have modulus a power of 2, as in \cite{Fortin}.

Recently, the second author \cite{Merca19a} provided a means for completely characterizating Ramanujan-type congruences modulo $16$ for $\overline{p}(n)$ by relating $\overline p(n)$ to the number of divisors of $n$. For any integer $k \geq 0$, we use the standard notation
\begin{align*}
\sigma_k(n) := \sum_{d | n} d^k
\end{align*}
to denote the sum of $k$th-powers of positive divisors of $n$, so the number of divisors of $n$ is given by $\sigma_0(n)$. By the proofs of \cite[Theorems 1.3 and 1.4]{Merca19a}, for $r = 3, 5$ and $n \geq 0$, we have
$$\overline{p}(8n+r) \equiv 0 \pmod{16}\quad \iff \quad \sigma_0(8n+r) \equiv 0 \pmod{4}.$$
If we let $\overline p_o(n)$ denote the number of overpartitions of $n$ into odd parts, the proof of the second authors also shows in Theorem 1 of \cite{Merca21} that for $r = 1,3$ and $n \geq 0$ we have
$$\overline{p_o}(8n+r) \equiv 0 \pmod{8}\quad \iff \quad \sigma_0(8n+r) \equiv 0 \pmod{4}.$$
These results may be viewed as steps towards classifying all Ramanujan-type congruences for overpartitions, particularly because the divisibility properties of multiplicative functions are more directly accessible with elementary methods than those of functions defined in terms of partitions. Motivated by these connections between overpartitions and divisor sums, the second author made the following conjecture in \cite{Merca22}.

\begin{conjecture}[Conjecture 7 of \cite{Merca22}] \label{CJ5}
	If $\sigma_0(An + B) \equiv 0 \pmod {2^{k}}$ holds for all $n\geqslant 0$, then 
	there is a sequence of prime numbers $p_1\leqslant p_2\leqslant \ldots \leqslant p_{k-1}$ such that
	$A$ is divisible by $(p_1p_2\cdots p_{k-1})^2$ and $B$ 
	is divisible by $p_1p_2\cdots p_{k-1}$.
\end{conjecture}

Because $\sigma_0(n)$ is a multiplicative function, this conjecture motivates the question of identifying all Ramanujan-type congruences for multiplicative functions. As stated this question is far too broad to be meaningfully resolved in any explicit way. However, there is a sense in which the fully general problem may be answered. Namely, since a multiplicative function is defined by its values at prime powers, the problem boils down to understanding how the divisibility properties of $f(n)$ at prime powers intersect with arithmetic progressions. Our first main objective is to make this relationship explicit.

In order to state our main result, we now fix some convenient notation which we will use throughout. For a fixed arithmetic progression $An+B$, we shall always define $G := \gcd(A,B)$, $A' := \frac{A}{G}$, and $B' := \frac{B}{G}$. We shall also denote by $G'$ the maximal divisor of $G$ such that $G' | A'$. For any prime $p$, we also let $\nu_p(m)$ denote the usual $p$-adic valuation, that is $\nu_p(m)$ is the multiplicity of $p$ as a factor of $m$. With this notation, our main theorem may be stated as follows.

\begin{theorem} \label{MAIN}
Let $f(n)$ be a integer-valued multiplicative function, $p$ any prime and $A, B > 0$ integers. Then the congruence
\begin{align*}
    f\lp An+B \rp \equiv 0 \pmod{p^k}
\end{align*}
holds for all $n \geq 0$ if and only if
\begin{align*}
    k \leq \nu_p\lp f(G') \rp + \sum_{\substack{q | G \\ q \centernot | A'}} U_p\lp A', B', \nu_q(G), q; f \rp + M_p\lp A', B', G; f \rp,
\end{align*}
where sums are taken over primes $q$, 
\begin{align} \label{U Definition}
U_p\lp A,B,a;q;f \rp := \min\limits_{n \geq 0} \nu_p\lp f\lp q^{a + \nu_q\lp An+B \rp} \rp \rp,
\end{align}
and
\begin{align} \label{M Definition}
M_p\lp A', B', C; f \rp := \min\limits_{n \geq 0} \sum_{q \centernot | C} \nu_p\lp f\lp q^{\nu_q(An+B)} \rp \rp.
\end{align}
\end{theorem}

Although the theorem uses a lot of notation, each term has a very explicit meaning which can be seen from their definitions. The value $\nu_p\lp f(G') \rp$ encapsulates contributions from factors of $An+B$ which always divide $An+B$ with the same order, $U_p$ encapsulates contributions coming from factors which always divide $An+B$ but to different orders for different $n$, and $M_p$ accounts for the behavior of the underlying arithmetic progression $A'n + B'$. This theorem may therefore be viewed as reducing the classification of Ramanujan-type congruences for a multiplicative function to progressions $An+B$ with coprime values of $A,B$. In practice, the most difficult term in Theorem \ref{MAIN} to control is the value of $M_p$ since the other two are more directly related to values of $f(n)$ at prime powers. We shall see, however, that in specific cases a great deal can be said about $M_p$. The key idea we utilize to study $M_p$ is that since $A'n + B'$ is infinitely often prime and prime values of $n$ are often those for which $f(n)$ has few prime divisors, and so behaviour of $f(n)$ at primes plays a central role in dictating the value of $M_p$.

After we prove Theorem \ref{MAIN}, the second main objective of this paper is to show the consequences of the main result in interesting cases. As the conjecture of the second author is the driving force behind the study, our first application of Theorem \ref{MAIN} is to the function $\sigma_k(n)$, which leads to the following corollary.

\begin{corollary} \label{MAIN COROLLARY}
Conjecture 1 is true.
\end{corollary}

The structure of the remainder of the paper is as follows. In Section 2, we prove Theorem \ref{MAIN} and a useful corollary which aids in explicit calculations. In Section 3, we investigate the consequences of Theorem \ref{MAIN} for three prominent cases, namely the divisor sums $\sigma_k(n)$, Euler's totient function $\varphi(n)$, and Ramanujan's tau function $\tau(n)$.

\section*{Acknowledgements}

The first author thanks the support of the grants of his Ph.D advisor Ken Ono, namely the Thomas Jefferson Fund and the NSF (DMS-1601306 and DMS-2055118).

\section{Proof of theorem \ref{MAIN}}

This section is dedicated to proving general results on divisibility properties of multiplicative functions on arithmetic progressions. We begin with the proof of Theorem \ref{MAIN}. After this proof, we prove a useful lemma which sheds light on the sorts of multiplicative functions whose behaviour in arithmetic progressions is particularly well-behaved.

\subsection{Proof of Theorem \ref{MAIN}}

For the proof of Theorem \ref{MAIN}, it is convenient to reframe the quality of being a Ramanujan-type congruence in terms of $p$-adic valuations. Recall that for $p$ prime and $a$ any positive integer, the value $\nu_p(a)$ denotes the number of times $a$ is divisible by $p$, so that
$$a = \prod_{p | a} p^{\nu_p(a)}.$$
We also say that $\nu_p(0) = +\infty$ for every prime $p$. To say that $f(An+B) \equiv 0 \pmod{p^k}$ is then to say that $\nu_p\lp f(An+B) \rp \geq k$. To describe a Ramanujan-type congruence in the language of valuations, we define the function $V_p\lp A,B;f \rp$ which collects all values of $\nu_p\lp f(An+B) \rp$ into a single number:

\begin{definition}
Let $p$ be a prime, $A, B > 0$ integers. For any integer-valued function $f$, we define
\begin{align} \label{V Definition}
    V_p\lp A, B; f \rp := \min\limits_{n \geq 0} \nu_p\lp f(An+B) \rp.
\end{align}
\end{definition}

We set $V_p(A,B; f) = +\infty$ if $f(An+B) = 0$ for all $n$, and otherwise $V_p\lp A, B; f \rp$ is clearly a well-defined non-negative integer. The lemma below follows immediately by definitions.
\begin{lemma} \label{Equivalence Lemma}
Let $A,B > 0$ be integers, $f(n)$ be a multiplicative arithmetic function and $p$ any prime. Then $f(An+B) \equiv 0 \pmod{p^k}$ for all $n$ if and only if $k \leq V_p\lp A, B; f \rp$.
\end{lemma}

Thus, locating Ramanujan-type congruences for $f(An+B)$ modulo $p^k$ is equivalent to evaluating $V_p\lp A, B; f \rp$. We now recall the definitions of $U_p\lp A, B, a; q; f \rp$ and $M_p\lp A, B, C; f \rp$ which appear in Theorem \ref{MAIN}. As with $V_p(A,B;f)$ we set values to $+\infty$ there are no $n$ for which $\nu_p\lp f\lp q^{a + \nu_p\lp An+B \rp} \rp \rp$ or $\nu_p\lp f\lp q^{\nu_q(An+B)} \rp \rp$ are finite. Informally, the function $M_p$ counts the contributions from $A'n + B'$ and $U_p$ and $\nu_p\lp f(G') \rp$ together count the contributions coming from $G := \gcd(A,B)$. Recall also the notations $A' := \frac{A}{G}$, $B' := \frac{B}{G}$, and $G'$ as the maximal divisor of $G$ such that $G' | A'$. Note that by Lemma \ref{Equivalence Lemma}, Theorem \ref{MAIN} is equivalent to proving that
\begin{align} \label{Key Identity}
V_p\lp A, B; f \rp = \nu_p\lp f(G') \rp + \sum_{\substack{q | G \\ q \centernot | A'}} U_p\lp A', B', \nu_q(G), q; f \rp + M_p\lp A', B', G; f \rp.
\end{align}

We will now prove the equality above directly. As in Theorem \ref{MAIN}, we assume $f(n)$ is a multiplicative arithmetic function, $p$ any prime and $A, B > 0$ are integers. For any integer $n \geq 0$, we have
\begin{align*}
\nu_p\lp f(An+B) \rp 
= \sum_{q | An+B} \nu_p\lp f\lp q^{\nu_q(G) + \nu_q(A'n+B')} \rp \rp.
\end{align*}
In order to minimize $\nu_p\lp f\lp An+B \rp \rp$, it is convenient to split off primes $q$ which divide $G$ and those which do not. Since $A'n + B' \equiv 0 \pmod{q}$ has a solution $n$ if and only if $q \centernot | A'$, when $q | A'$ we have $\nu_q(A'n + B') = 0$ for all $n$ and therefore
\begin{align*}
\sum_{\substack{q | G \\ q | A'}} \nu_p\lp f\lp q^{\nu_q(G) + \nu_q(A'n+B')} \rp \rp = \sum_{\substack{q | G \\ q | A'}} \nu_p\lp f\lp q^{\nu_q(G)} \rp \rp = \nu_p\lp f(G') \rp.
\end{align*}
For primes $q | G$ with $q \centernot | A'$, $A'n + B' \equiv 0 \pmod{q}$ has a solution in $n$. Furthermore, by Hensel's lemma we can find congruence classes $n \equiv n_{q^k} \pmod{q^{k+1}}$ such that $A'n + B' \equiv 0 \pmod{q^k}$ and $A'n + B' \not \equiv 0 \pmod{q^{k+1}}$ for all such $n$. For each $q | G$, $q \centernot | A'$ we choose a congruence class modulo some power of $q$ such that for all $n$ in this congruence class, $\nu_p\lp f\lp q^{\nu_q(G) + \nu_q(A'n+B')} \rp \rp = U_p\lp A', B', \nu_q(G), q; f \rp$. By the Chinese remainder theorem, we may impose these restrictions on $n$ simultaneously for all $q | G, q \centernot | A'$, so that for $n$ on some arithmetic progression we have
\begin{align*}
\min\limits_{n \geq 0} \sum_{\substack{q | G \\ q \centernot | A'}} \nu_p\lp f\lp q^{\nu_q(G) + \nu_q(A'n+B')} \rp \rp = \sum_{\substack{q | G \\ q \centernot | A'}} U_p\lp A', B', \nu_q(G), q; f \rp.
\end{align*}
Therefore, decomposing the sum over $q | An+B$ appropriately, we have
\begin{align*}
V_p\lp A, B; f \rp = \nu_p\lp f(G') \rp + \sum_{\substack{q | G \\ q \centernot | A'}} U_p\lp A', B', \nu_q(G), q; f \rp + M_p\lp A', B', G; f \rp
\end{align*}
by the preceding arguments and the definition of $M_p(A', B', G; f)$. This now completes the proof of Theorem \ref{MAIN} by Lemma \ref{Equivalence Lemma}.

\subsection{A Useful Consequence}

Applying Theorem \ref{MAIN} to particular multiplicative functions $f(n)$ boils down to computing values of $U_p$ and $M_p$. Informally, evaluating $U_p$ requires understanding sequences of the form $\{ f(q^m) \}_{m \geq a}$ and evaluating $M_p$ requires identifying values of $n$ where $\nu_p\lp f(A'n+B') \rp$ is small. The following lemma lays out some fairly straightforward elementary conditions for which the values of $U_p$ and $M_p$ are easy to evaluate and which occur in many applications.

\begin{lemma} \label{U and M Evaluation}
Let $A, B, C \geq 1$, be integers such that $\gcd(A,B) = 1$, and let $p,q$ be primes. For any integer-valued multiplicative function $f(n)$, the following are true: \\

\noindent (a) If for $q$ prime we have $\nu_p\lp f(q^k) \rp \leq m$ for infinitely many $k$, then $U_p\lp A, B, a, q; f \rp \leq m$ for every $a \geq 0$. \\

\noindent (b) If $\nu_p(f(q)) \leq m$ for infinitely many primes $q \equiv B \pmod{A}$, then $M_p(A, B, C; f) \leq m$.
\end{lemma}

\begin{proof}
For any $n$, $M_p(A, B, C; f) \leq \nu_p\lp f(An+B) \rp$ always holds. By Hensel's lemma, the equation $m = a + \nu_q\lp An+B \rp$ has a solution $n$ for every $m \geq a$. Choosing $k \geq a$ such that $\nu_p\lp f(q^k) \rp \leq m$, we therefore have $U_p\lp A, B, a, q; f \rp \leq m$. Part (b) follows since there are infinitely many primes of the form $An+B$, so in particular there is some prime $q$ of this form not dividing $C$, from which $M_p(A, B, C; f) \leq m$ follows.
\end{proof}

\section{Important Special Cases}

We now investigate several important special cases of the results above. The proofs below use \eqref{Key Identity} implicitly throughout.

\subsection{Divisor Sums} \label{Divisor Sums Section}

We now apply our results to the function $\sigma_k(n) = \sum_{d|n} d^k$. We prove examples of interesting congruence properties for $\sigma_k(n)$. We first study $\sigma_0(n)$, because it behaves quite differently from other $\sigma_k(n)$ and because this case directly pertains to Conjecture \ref{CJ5}.

\begin{corollary} \label{Sigma 0 Corollary}
For any integers $A,B \geq 1$ and any prime $p$, we have $V_p(A, B; \sigma_0) \leq \nu_p\lp \sigma_0\lp G' \rp \rp + 1$. Furthermore, if $V_p(A, B; \sigma_0) = \nu_p\lp \sigma_0\lp G' \rp \rp + 1$, then $p = 2$ and $\legendre{B'}{A'} = - 1$.
\end{corollary}

\begin{proof}
First, assume that $p > 2$. Since for $q = A'n + B'$ prime we have $\nu_p(\sigma_0(A'n + B')) = 0$, it follows by Lemma \ref{U and M Evaluation} (b) that $M_p(A', B',G; \sigma_0) = 0$. Similarly, for $q$ prime we have $\nu_p(\sigma_0(q^\ell)) > 0$ if and only if $\ell \equiv -1 \pmod{p}$, and thus from Lemma \ref{U and M Evaluation} (a) we have $U_p(A', B', \nu_q(G), q; \sigma_0) = 0$ for all $q$. Therefore, by Theorem \ref{MAIN} it follows that $V_p(A, B; \sigma_0) = \nu_p\lp \sigma_0\lp G' \rp \rp$.

Now, we consider the case $p = 2$. The argument that $U_2(A', B', \nu_q(G), q; \sigma_0) = 0$ still holds, so we need only consider the values of $M_2(A', B', G; \sigma_0)$. By Lemma \ref{U and M Evaluation} (b), we have immediately that $M_2(A', B', G; \sigma_0) \leq 1$. If $\legendre{B'}{A'} = 1$, then $A'n + B'$ is is infinitely often a perfect square and so in this case $M_2(A', B', G; \sigma_0) = 0$, which completes the proof.
\end{proof}


Corollary \ref{Sigma 0 Corollary} is in fact now sufficient to prove Conjecture \ref{CJ5}.

\begin{proof}[Proof of Conjecture \ref{CJ5}]
Suppose $\sigma_0(An+B) \equiv 0 \pmod{2^k}$ for all $n \geq 0$. Without loss of generality, we may assume that $k>1$ is chosen so that $V_2\lp A, B; \sigma_0 \rp = k$. By Corollary \ref{Sigma 0 Corollary} we then have $0 < k - 1 \leq \nu_2\lp \sigma_0(G') \rp$. If we write $G' = q_1^{a_1} \cdots q_\ell^{a_\ell}$ for distinct primes $q_i$, then
$$\nu_2\lp \sigma_0\lp G' \rp \rp = \sum_{i=1}^\ell \lp 1 + a_i \rp \geq k-1 > 0.$$
Note that the number of prime divisors of $G'$ counting multiplicity is therefore at least $k-1$. Therefore we may define the number $P := p_1 \cdots p_{k-1}$ for primes $p_1 \leq  p_2 \leq \cdots \leq p_{k-1}$ such that each $p_j$ is equal to some $q_i$ and each $q_i$ occurs with multiplicity at most $a_i$ among the $p_j$. Because $P | G$ we have $P | A$ and $P | B$, and since $P | G'$ we have $P | A'$, so $P^2 | A$. Therefore $P | B$ and $P^2 | A$, which completes the proof.
\end{proof}

For general $\sigma_k(n)$, we can also place certain strict conditions on $A', B'$ which govern the possibility of Ramanujan-type congruences modulo 2.

\begin{corollary}
For integers $A,B,k \geq 1$, we have that $V_2(A, B, G; \sigma_k) \geq \nu_2\lp \sigma_k\lp G' \rp \rp$, and furthermore that $V_2(A, B, G; \sigma_k) > \nu_2\lp \sigma_k \lp G' \rp \rp$ only if $\legendre{2}{A'} = 1$ and $\legendre{B'}{A'} = - 1$.
\end{corollary}

\begin{proof}
For any prime $q$, $\sigma_k(q^m)$ is even if and only if $q > 2$ and $m$ is odd. By Lemma \ref{U and M Evaluation} (a), it therefore follows that for every prime $q$ we have $U_2(A', B', a, q; \sigma_k) = 0$ for any $a \geq 0$. By Theorem \ref{MAIN}, we therefore have $V_2(A,B;\sigma_k) = \nu_2\lp \sigma_k\lp G' \rp \rp + M_2(A', B', G; \sigma_k)$.

We turn now to evaluating $M_2(A', B', G; \sigma_k)$, beginning with $p=2$. It is easy to see that $\sigma_k(n)$ is odd if and only if $n = m^2$ or $n = 2 m^2$ for some integer $m$. So, if $B'$ is congruent to either a square or twice a square modulo $A'$, then $n$ may be chosen so that $A'n + B' = m^2$ or $2m^2$, and therefore $M_2(A',B',G; \sigma_k) = 0$ if $B'$ is a square or twice a square modulo $A'$. Because $\legendre{2B'}{A'} = \legendre{2}{A'} \legendre{B'}{A'}$, $B'$ fails both these conditions if and only if $\legendre{2}{A'} = 1$ and $\legendre{B'}{A'} = -1$.
\end{proof}

For $A,B$ which are not covered by the above result, a practical classification of congruences appears to be beyond the reach of these methods. However, we can give many explicit congruences of this type by making use of the properties of quadratic residues and nonresidues. We discuss several of these examples involving small primes in the result below.

\begin{corollary} \label{Sigma Special k}
Let $k$ be a positive integer. The following are true:

\noindent (a) If $k \equiv 1 \pmod{2}$, we have $V_2(4,3,\sigma_k) = 2$. In particular, we have $$\sigma_k(4n+3) \equiv 0 \pmod{4}$$ for all $n$. \\

\noindent (b) If $k \equiv 1 \pmod{2}$, we have $V_2(8,7,\sigma_k) = 3$, $V_2(8,5;\sigma_k) = 1$, and $V_2(8,3;\sigma_k) = 2$. In particular, the three congruenes $$\sigma_k(8n+7) \equiv 0 \pmod{8}, \hspace{1cm} \sigma_k(8n+5) \equiv 0 \pmod{2}, \hspace{1cm} \sigma_k(8n+3) \equiv 0 \pmod{4}$$ each hold for all $n$. \\

\noindent (c) If $k \geq 2$ satisfies $k \equiv 0 \pmod{2}$, we have $V_3(3,2;\sigma_k) = 1$, that is, $$\sigma_k(3n+2) \equiv 0 \pmod{3}$$ holds for all $n$. \\

\noindent (d) If $k \equiv 1 \pmod{2}$, we have $V_5(5,2;\sigma_k) = V_5(5,3;\sigma_k) = 1$. In particular, we have
$$\sigma_k(5n+2) \equiv \sigma_k(5n+3) \equiv 0 \pmod{5}$$ for all $n$.

\noindent (e) If $k \equiv 3 \pmod{6}$, we have $V_7(7,b;\sigma_k) \geq 1$, for $b \in \{ 3, 5, 6 \}$. That is, $$\sigma_k(7n+b) \equiv 0 \pmod{7}$$ for all $n$ for $b \in \{ 3, 5, 6 \}$.
\end{corollary}

\begin{proof}
Since $4n+3$ is never a perfect square, the divisors $d, \frac{4n+3}{d}$ of $4n+3$ are distinct and $d \not \equiv \frac{4n+3}{d} \pmod{4}$. Since $k$ is odd, $d^k \not \equiv \lp \frac{4n+3}{d} \rp^k \pmod{4}$, and therefore $d^k + \lp \frac{4n+3}{d} \rp^k \equiv 0 \pmod{4}$. Thus, $\sigma_k(4n+3) \equiv 0 \pmod{4}$ for all $n$, in other words $V_2(4,3;\sigma_k) \geq 2$. Since $k$ is odd, $\sigma_k(3) = 1 + 3^k \equiv 4 \pmod{8}$ and therefore $V_2\lp 4, 3; \sigma_k \rp = 2$, which completes the proof of (a). The proofs of (b)-(e) are similar, each using elementary facts about quadratic residues and nonresidues to prove lower bounds on $V_p$ and small values of $n$ to make the lower bound into an equality.
\end{proof}

A similar approach can be used for $\sigma_k(n)$ for arithmetic progressions $A'n + B'$ with $A'$ prime and $B'$ a quadratic nonresidue modulo $A'$ with $k$ chosen suitably to line up with Euler's criterion for quadratic residues. We also note that congruences for $\sigma_k(n)$ on arithmetic progressions have connections with representations of integers as quadratic forms arising as norms of number fields. We prove the simplest example of such a result, namely the case of $\mathbb{Q}(i)$.

\begin{corollary}
Suppose that $n$ is not the sum of two squares. Then $\sigma_k(n) \equiv 0 \pmod{4}$.
\end{corollary}

\begin{proof}
The positive integer $n$ is not a sum of two squares if and only if $n$ has a prime divisor $p \equiv 3 \pmod{4}$ occurring with odd multiplicity $\ell$. By multiplicativity, $\sigma_k(p^{\ell})$ divides $\sigma_k(n)$, so it suffices to show that $\sigma_k(p^\ell) \equiv 0 \pmod{4}$. But if $\ell = 2m + 1$ is odd, we have
\begin{align*}
\sigma_k(p^\ell) = 1 + p + \dots + p^{2m+1} = \lp 1 + p^2 + \dots + p^{2m} \rp \lp 1 + p \rp \equiv 0 \pmod{4}
\end{align*}
since $p \equiv 4 \pmod{3}$. This completes the proof.
\end{proof}

\subsection{Euler's $\varphi(n)$}

Another famous multiplicative function is Euler's $\varphi(n)$, which counts the number of residue classes modulo $n$ which are coprime to $n$. The multiplicative function $\varphi(n)$ is defined by its values at primes powers, which are $\varphi(p^k) = p^{k-1}(p-1)$.

\begin{prop}
Let $A, B > 0$ be integers and $p$ a prime. If $B' \not \equiv 1 \pmod{p}$, we have
\begin{align*}
V_p\lp A,B; \varphi \rp = \nu_p\lp \varphi\lp G' \rp \rp + \sum_{\substack{q | G \\ q \centernot | A'}} \nu_p(q-1) + \begin{cases} \nu_p\lp \varphi \lp G \rp \rp - 1 & \text{ if } p | G, p \centernot | A', \\ 0 & \text{ otherwise}. \end{cases}
\end{align*}
\end{prop}

\begin{proof}
For any prime $q$, there are values of $n$ for which $q \centernot | A'n + B'$ and therefore $U_p(A', B', 0, q; \varphi) = 0$. If $a>0$, we furthermore have
\begin{align*}
U_p(A', B', a, q; \varphi) = \min_{n \geq 0} \nu_p\lp \varphi\lp q^{a + \nu_q(A'n + B')} \rp \rp = \nu_p\lp \varphi(q^a) \rp = \begin{cases} a-1 & \text{ if } p = q, \\ \nu_p(q-1) & \text{ if } p \not = q. \end{cases}
\end{align*}
Therefore, since $\nu_p(p-1) = 0$,
\begin{align*}
\sum_{\substack{q | G \\ q \centernot | A'}} U_p\lp A', B', \nu_q(G), q; \varphi \rp = \sum_{\substack{q | G \\ q \centernot | A'}} \nu_p(q-1) + \begin{cases} \nu_p\lp \varphi\lp G \rp \rp - 1 & \text{ if } p | G, p \centernot | A', \\ 0 & \text{ otherwise}. \end{cases}
\end{align*}
It remains to calculate $M_p(A', B', G; \varphi)$. Now, by \eqref{M Definition} and the known values of $\varphi$ at prime powers, we have $M_p(A', B', G; \varphi) = 0$ if and only if there exists some $n$ for which $p^2 \centernot | A'n + B'$ and every prime $q | A'n + B'$, $q \centernot | G$ satisfies $q \not \equiv 1 \pmod{p}$. If $B' \not \equiv 1 \pmod{p}$, then by Dirichlet's theorem we may select $n = pm$ so that $A'n + B'$ is a prime not dividing $G$. Therefore, $M_p(A',B',G;\varphi) = 0$. Combining these results with those for $U_p$, by Theorem \ref{MAIN} we conclude that if $B' \not \equiv 1 \pmod{p}$, the result follows.
\end{proof}

Beyond the scope of this result, there do not seem to be any `nontrivial' Ramanujan-type congruences for $\varphi(n)$, although our methods cannot address this question. We therefore formulate the following conjecture for $\varphi(n)$, which may be viewed as the claim that $\varphi(n)$ is a multiplicative function having no ``nontrivial'' Ramanujan-type congruences.

\begin{conjecture} 
When $p>2$ is a prime, there are no coprime integers $A', B'$ such that $\varphi(A'n+B') \equiv 0 \pmod p$ for all $n\geq 0$. 
\end{conjecture}	

\subsection{Fourier Coefficients of Hecke Eigenforms}

Apart from the elementary multiplicative functions $\sigma_k(n)$ and $\varphi(n)$, one of the most important families of multiplicative functions are the coefficients of cuspidal Hecke eigenforms. For the sake of simplicity, we shall focus only on the important case of Ramanujan's $\tau$-function $\tau(n)$, which is given as the coefficients of the infinite product
\begin{align*}
\sum_{n=1}^\infty \tau(n) q^n := q \prod_{n=1}^\infty (1 - q^n)^{24} = q - 24 q^2 + 252 q^3 - 1472 q^4 + 4830 q^5 + \cdots.
\end{align*}
Mordell \cite{Mordell} established that the function $\tau(n)$ is multiplicative and that for primes $p$, the sequence of values $\tau(p^m)$ for $m \geq 0$ satisfy the recurrence relation
\begin{align*}
\tau(p^{m+1}) = \tau(p) \tau(p^m) - p^{11} \tau(p^{m-1}).
\end{align*}
Note that this recurrence relation is itself a useful tool for proving certain kinds of congruence properties. For example, from the fact that $\tau(2) = 24$ we may deduce from the recursive formula above that $\tau(2^k) \equiv 0 \pmod{8}$ for all $k \geq 1$. In fact, once it is also known that $\tau(p)$ is even for all primes $p$, these recurrences show that $\tau(n)$ is odd only when $n$ is an odd square.

We may therefore in principle use Theorem \ref{MAIN} to study congruences for $\tau(n)$. The situation is complicated somewhat by the fact that Lehmer's conjecture \cite{Lehmer} that $\tau(n)$ never vanishes remains open, so it is not yet possible to guarantee in full generality that $V_p(A,B;\tau)$ is finite. Furthermore, writing down the values of $\tau(p)$ is not nearly as simple as with many other arithmetic functions, and so the direct application of Theorem \ref{MAIN} to $\tau(n)$ is not as fruitful. However, $\tau(n)$ satisfies many exceptional congruences \cite{Swinnterton-Dyer} related to divisor sums $\sigma_k(n)$ with various odd values of $k$, and as we have seen the study congruences for $\sigma_k(n)$ is accessible with even elementary methods. For the primes 2 and 3, we have the following sets of congruence relations:
\begin{align*}
\tau(8n+1) &\equiv \sigma_{11}(8n+1) \pmod{2^{11}}, \\
\tau(8n+3) &\equiv 1217 \sigma_{11}(8n+3) \pmod{2^{13}}, \\
\tau(8n+5) &\equiv 1537 \sigma_{11}(8n+5) \pmod{2^{12}}, \\
\tau(8n+7) &\equiv 705 \sigma_{11}(8n+7) \pmod{2^{14}}, \\
\tau(3n+1) &\equiv n^{-610} \sigma_{1231}(3n+1) \pmod{3^6}, \\
\tau(3n+2) &\equiv n^{-610} \sigma_{1231}(3n+2) \pmod{3^7}. \\
\end{align*}
For other small primes, we have the following additional congruences
\begin{align*}
\tau(n) &\equiv n^{-30} \sigma_{71}(n) \pmod{5^3} \text{ for } n \not \equiv 0 \pmod{5}, \\
\tau(n) &\equiv n \sigma_9(n) \pmod{7} \text{ for } n \equiv 0, 1, 2, 4 \pmod{7}, \\
\tau(n) &\equiv n \sigma_9(n) \pmod{7} \text{ for } n \equiv 3, 5, 6 \pmod{7^2}.
\end{align*}
We may therefore translate our understanding of Ramanujan-type congruences for $\sigma_k(n)$ into like congruences for $\tau(n)$. For example, by Corollary \ref{Sigma Special k} (e) the connection to $\sigma_9(n)$ reproves Ramanujan's observation stated in \cite{Ramanujan2} that
$$\tau(7n+3) \equiv \tau(7n+5) \equiv \tau(7n+6) \equiv 0 \pmod{7}.$$
From (a)-(d) of Corollary \ref{Sigma Special k}, we also obtain the analogous congruences for $\tau(n)$ modulo powers of other small primes.

We note briefly that Ramanujan's tau function is special in its direct connection to divisor sums, but not in multiplicativity. More generally, if $f(q) = \sum_{n \geq 0} a_f(n) q^n$ is any Hecke eigenform, then the coefficients $a_f(n)$ are a multiplicative function of $n$ with the additional nice property that for each prime $p$, the values $a_f\lp p^m \rp$ for $m \geq 2$ are determined by the value of $a_f(p)$ via a two-term linear recurrence relation analogous to the classical Fibonacci sequence. The philosophy which arises from Theorem \ref{MAIN} says that all Ramanujan-type congruences of $a_f(n)$ are induced by the values $a_f(p)$. As much is known about the multiplicative structure of these types of recurrence relations, it would be interesting to pursue a classification of Ramanujan-type congruenes for the coefficients of eigenforms along these lines.

\end{document}